\documentclass[10pt]{amsart}
\usepackage{enumerate,amsmath,amssymb,latexsym,
amsfonts, amsthm, amscd, MnSymbol, diagrams}

\theoremstyle{plain}

\newtheorem{theorem}{Theorem}

\numberwithin{equation}{section}

\newcommand{\ra}{\rightarrow}

\newcommand{\CC}{\mathbb{C}}

\addtocounter{section}{-1}

\begin{document}

\title {Berezin expectation and Clifford trace}

\date{}

\author[P.L. Robinson]{P.L. Robinson}

\address{Department of Mathematics \\ University of Florida \\ Gainesville FL 32611  USA }

\email[]{paulr@ufl.edu}

 \keywords{}

\begin{abstract}

We show that in finite dimensions, the Berezin expectation on an exterior algebra corresponds to the trace on a Clifford algebra, the correspondence being mediated by what is effectively a version of antinormal ordering. 

\end{abstract}

\maketitle

\medbreak

\section*{Introduction}

\medbreak 

Berezin [1] introduced a `calculus' in which to frame the discussion of fermionic systems along lines that are remarkably similar to those involved in the standard discussion of bosonic systems. The details of this calculus were extracted and elaborated in invariant terms in [4]: among other things, to any (finite-dimensional) complex inner product space $V$ is associated a Berezin expectation $\mathbb{E} : \bigwedge (V_{\CC}) \ra \CC$, this being a linear functional defined on the exterior algebra of the complexified underlying real inner product space; among its many properties, this Berezin expectation facilitates construction of the standard (determinantal) inner product on the exterior algebra $\bigwedge (V)$ of $V$ itself. Quite separately, the  real inner product space underlying $V$ has an associated complex Clifford algebra $C(V_{\CC})$: among its many structural features this superalgebra carries a unique trace $\tau : C(V_{\CC}) \ra \CC$, this being a normalized linear functional that kills odd elements and satisfies the tracial identity $\tau (a b) = \tau (b a)$; [3] is a convenient reference for this. 

\medbreak 

Our aim here is to demonstrate that these two structures - the Berezin expectation $\mathbb{E}$ and the Clifford trace $\tau$ - are intimately related. In fact, we show that the complex structure of $V$ gives rise to a (super) linear isomorphism $\nu : \bigwedge (V_{\CC}) \ra C(V_{\CC})$ such that $\mathbb{E} = \tau \circ \nu$; the symbol $\nu$ is chosen because this map amounts to (anti-)normal ordering. As indicated above, the fundamental vector space is here taken to be finite-dimensional; the infinite-dimensional case will be treated in a subsequent paper.   

\medbreak 

\section*{Expectation and Trace} 

\medbreak 

Let $V$ be a real vector space upon which $( \bullet | \bullet )$ is an inner product. This inner product extends to the complexification $V_{\CC}$ both as a complex-bilinear form $(\bullet | \bullet)$ and as a complex inner product $\langle \bullet | \bullet \rangle$; these are related by 
$\langle x | y \rangle = ( \bar{x} | y)$
where an overbar signifies the canonical conjugation in $V_{\CC}$. We shall denote by $C(V_{\CC})$ the complex Clifford algebra over the complex vector space $V_{\CC}$ equipped with the bilinear form $( \bullet | \bullet )$. Explicitly, $C(V_{\CC})$ is generated by its subspace $V_{\CC}$ subject to the relations $z^2 = (z | z) {\bf 1}$ for every $z \in V_{\CC}$; by polarization, if $x, y \in V_{\CC}$ then 
$x y + y x = 2 (x | y) {\bf 1}.$ 
The algebra $C(V_{\CC})$ carries a number of natural structures: a unique automorphism $\Gamma$ that restricts to $V_{\CC} \subseteq C(V_{\CC})$ as minus the identity and makes $C(V_{\CC})$ into a superalgebra; a unique involution $^*$ that restricts to $V_{\CC} \subseteq C(V_{\CC})$ as the canonical conjugation; and a unique linear functional $\tau : C(V_{\CC}) \ra \CC$ that satisfies $\tau ({\bf 1}) = 1$, is $\Gamma$-invariant and has the tracial property $\tau (a b) = \tau (b a)$ for each $a, b \in C(V_{\CC})$. The involution $^*$ and the trace $\tau$ together equip $C(V_{\CC})$ with a complex inner product defined by the rule that if $a, b \in C(V_{\CC})$ then 
$\langle a | b \rangle = \tau (a^* b).$ 
The fact that we write $\langle \bullet | \bullet \rangle$ for two complex inner products here should not cause confusion: indeed, the tracially-defined inner product on $C(V_{\CC})$ restricts to the sesquilinear extension of $( \bullet | \bullet )$ on its subspace $V_{\CC}$. 

\medbreak 

Now, let $J: V \ra V$ be a $(\bullet | \bullet)$-orthogonal complex structure: $J$ makes $V$ into a complex vector space $V_J$ on which a complex inner product is defined by  
$\langle x | y \rangle_J = (x | y) + i (J x | y).$
The complexification $J_{\CC} : V_{\CC} \ra V_{\CC}$ engenders an eigendecomposition 
$$V_{\CC} = V_J^- \oplus V_J^+$$
where 
$$V_J^- = \{ w \in V_{\CC} : J_{\CC} w = - i w \} $$
and 
$$V_J^+ = \{ w \in V_{\CC} : J_{\CC} w = + i w \}.$$
Each of these eigenspaces is totally isotropic for the form $( \bullet | \bullet )$ on $V_{\CC}$; moreover, the eigenspaces are perpendicular for the inner product $\langle \bullet | \bullet \rangle$ on $V_{\CC}$. In fact, the map 
$$\gamma^+ : V_J \ra V_J^+ : v \mapsto v^+ = (v - i J v)/\sqrt{2}$$
is a unitary isomorphism from $V_J$ with $\langle \bullet | \bullet \rangle_J$ to $V_J^+ \subseteq V_{\CC}$ with $\langle \bullet | \bullet \rangle$ while 
$$\gamma^- : V_J \ra V_J^- : v \mapsto v^- = (v + i J v)/\sqrt{2}$$
is an antiunitary isomorphism in the corresponding sense. Notice that if $v \in V_J$ then $v^-$ is the conjugate of $v^+$ and that if $x, y \in V_J$ then 
$$(x^- | y^+) = \langle x | y \rangle_J.$$We extend $\gamma^+$ functorially to an isomorphism of exterior algebras: thus 
$$\gamma^+ : \bigwedge(V_J) \ra \bigwedge(V_J^+)$$
is a unitary isomorphism when the exterior algebras are given their standard (determinantal) inner products; similarly, we extend $\gamma^-$ functorially to an antiunitary {\it anti}-isomorphism 
$$\gamma^- : \bigwedge(V_J) \ra \bigwedge(V_J^-).$$

\medbreak 

The Berezin expectation $\mathbb{E} : \bigwedge (V_{\CC}) \ra \CC$ was introduced in [4] and inspired by the classical text [1]; the fundamentals may be summarized as follows. The sum 
$$\gamma = \sum_{m = 1}^M \gamma^+(v_m) \wedge \gamma^- (v_m) \in \bigwedge (V_{\CC})$$
is independent of the unitary basis $v_1, \dots , v_M$ for $V_J$ and 
$$\omega = \frac{1}{M!} (- \gamma)^M = \gamma^-(v_1) \wedge \gamma^+ (v_1) \wedge \cdots \wedge \gamma^-(v_M) \wedge \gamma^+ (v_M)$$
is a preferred unit vector relative to the canonical inner product on the exterior algebra $\bigwedge (V_{\CC})$. The Berezin expectation $\mathbb{E} (\zeta) \in \CC$ of $\zeta \in \bigwedge (V_{\CC})$ is now defined by 
$$\mathbb{P} (\zeta \wedge e^{- \gamma} ) = \mathbb{E} (\zeta) \omega$$
where $\mathbb{P}$ is the orthogonal projector from $\bigwedge (V_{\CC})$ to its top exterior power $\bigwedge^{2 M} (V_{\CC})$ and where exponentiation takes place in the exterior algebra. For our present purposes, the most significant property of the linear functional $\mathbb{E}$ may be stated as follows. 

\medbreak 

\begin{theorem} \label{BerIP}
If $\xi, \eta \in \bigwedge (V_J)$ then $\mathbb{E} (\gamma^- (\xi) \wedge \gamma^+ (\eta)) = \langle \xi | \eta \rangle_J.$
\end{theorem} 

\begin{proof} 
See [4] Theorem 1.7 and the discussion thereafter.
\end{proof} 

\medbreak 

It is important to remark here that the decomposition $V_{\CC} = V_J^- \oplus V_J^+$ naturally induces an isomorphism 
$$\bigwedge (V_{\CC}) \ra \bigwedge (V_J^-) \otimes \bigwedge (V_J^+)$$
of superalgebras, the tensor product being super; the inverse isomorphism maps $\gamma^-(\xi) \otimes \gamma^+ (\eta)$ to $\gamma^-(\xi) \wedge \gamma^+ (\eta)$ whenever $\xi, \eta \in \bigwedge (V_J)$. For this, the reader may refer to Section 5 of Chapter V in [2]. It follows that each element of $\bigwedge (V_{\CC})$ is actually a finite linear combination of the wedge products $\gamma^- (\xi) \wedge \gamma^+ (\eta)$ that appear in this theorem. 

\medbreak 

The effect of the decomposition $V_{\CC} = V_J^- \oplus V_J^+$ on the Clifford algebras is different. The subalgebra of $C(V_{\CC})$ generated by any complex subspace $W$ of $V_{\CC}$ is canonically the Clifford algebra $C(W)$ of $W$ equipped with the restriction of $(\bullet | \bullet)$ as symmetric bilinear form. In particular, the subalgebra of $C(V_{\CC})$ generated by the isotropic subspace $V_J^+$ is the Clifford algebra of $V_J^+$ equipped with the identically zero form: that is, $C(V_J^+)$ is simply the exterior algebra $\bigwedge (V_J^+)$; likewise, $\bigwedge (V_J^-) = C(V_J^-) \subseteq C(V_{\CC})$. In general, each direct sum decomposition $V_{\CC} = X \oplus Y$ gives rise to an isomorphism 
$$C(X) \otimes C(Y) \ra C(V_{\CC}): a \otimes b \mapsto a b $$
of super vector spaces, the tensor product again being super and the action being indicated on decomposables; this is an isomorphism of superalgebras when the direct sum decomposition is $(\bullet | \bullet)$-orthogonal. For a discussion of such matters, see Section 3 of Chapter 5 in [5]. In the present situation, we arrive at a super vector space isomorphism 
$$\bigwedge (V_J^-) \otimes \bigwedge (V_J^+) \ra C(V_{\CC}) : \gamma^-(\xi) \otimes \gamma^+ (\eta) \mapsto \gamma^-(\xi) \gamma^+ (\eta).$$

\medbreak 

Composition now yields a $J$-dependent isomorphism of super vector spaces 
$$\nu = \nu_J :  \bigwedge (V_{\CC}) \ra C(V_{\CC}).$$

\medbreak

A couple of simple examples will illustrate the nature of this isomorphism. For the statements, we prefer to drop the inner product subscript $J$ and write $\langle \bullet | \bullet \rangle_J$ simply as $\langle \bullet | \bullet \rangle$.

\medbreak 

\begin{theorem} \label{nor2} 
If $x, y \in V_J$ then 
$$\nu (x \wedge y) = x y - \langle y | x \rangle {\bf 1}.$$
\end{theorem} 

\begin{proof} 
Observe that $\sqrt{2} \: x = x^- + x^+$ and $\sqrt{2} \: y = y^- + y^+$ whence 
\begin{eqnarray*}
2 \: ( x \wedge y) & = & x^- \wedge y^- + x^- \wedge y^+ + x^+ \wedge y^- + x^+ \wedge y^+ \\ & = & x^- \wedge y^- + x^- \wedge y^+ - y^- \wedge x^+ + x^+ \wedge y^+
\end{eqnarray*}
and therefore 
\begin{eqnarray*} 
2 \: \nu( x \wedge y ) & = & x^-  y^- + x^-  y^+ - y^-  x^+ + x^+  y^+ \\ & = & (x^- + x^+) (y^- + y^+) - x^+ y^- - y^- x^+ \\ & = & 2 x y \; - \; 2 (y^- | x^+ )\\ & = & 2 x y - 2 \langle y | x \rangle_J.
\end{eqnarray*}
\end{proof} 

\medbreak 

\begin{theorem} \label{nor3}
If $x, y, z \in V_J$ then 
$$\nu (x \wedge y \wedge z) = x y z - \langle z | y \rangle x + \langle z | x \rangle y - \langle y | x \rangle z.$$
\end{theorem} 

\begin{proof} 
This is entirely similar to the previous proof. After multiplication by $ 2 \sqrt{2}$, the difference $\nu (x \wedge y \wedge z) - x y z$ reduces to the sum of four terms that may be rewritten using the relations in the Clifford algebra, as
\begin{eqnarray*}
- y^- x^+ z^+ - x^+ y^- z^+ & = & - 2 \langle y | x \rangle z^+ \\
+ z^- x^+ y^+ - x^+ y^+ z^- & = & + 2 \langle z | x \rangle y^+ - 2 \langle z | y \rangle x^+ \\
+ y^- z^- x^+ - x^+ y^- z^- & = & + 2 \langle z | x \rangle y^- - 2 \langle y | x \rangle z^- \\
- x^- z^- y^+ - x^- y^+ z^- & = & - 2 \langle z | y \rangle x^-
\end{eqnarray*}
and these four terms sum correctly to yield the announced formula.  
\end{proof} 

\medbreak 

In effect, $\nu$ amounts to a version of antinormal ordering. 

\medbreak 

We are now in a position to show that under this version of antinormal ordering, the Berezin expectation on $\bigwedge (V_{\CC})$ corresponds to the trace on $C(V_{\CC})$. 

\medbreak 

\begin{theorem} \label{main} 
$\mathbb{E} = \tau \circ \nu.$
\end{theorem} 

\begin{proof} 
If $\xi, \eta \in \bigwedge (V_J)$ then $\nu (\gamma^-(\xi) \wedge \gamma^+ (\eta)) = \gamma^- (\xi) \gamma^+ (\eta)$ while 
$$\mathbb{E} (\gamma^-(\xi) \wedge \gamma^+ (\eta)) = \langle \xi | \eta \rangle_J = \langle \gamma^+(\xi) | \gamma^+ (\eta) \rangle = \tau (\gamma^+(\xi)^* \gamma^+ (\eta)) = \tau (\gamma^- (\xi) \gamma^+ (\eta)).$$
Here: the first equality comes from Theorem \ref{BerIP}; the second holds because $\gamma^+$ is a unitary isomorphism; the third holds by definition of the (tracial) inner product on the Clifford algebra; and the fourth holds as the involution converts $\gamma^+$ to $\gamma^-$. 
\end{proof} 

\medbreak 

In other words $\nu$, $\mathbb{E}$ and $\tau$ fit into the following commutative diagram: 

\begin{diagram}
\bigwedge (V_{\CC})         &\rTo_{\nu}   &C(V_{\CC})\\
\dTo_{\mathbb{E}}  &           &\dTo^{\tau}\\
\CC         &\rTo^{=}   &\CC
\end{diagram}

\bigbreak 

\medbreak 

We close with a couple of brief remarks. 

\medbreak 

For ease of reference, our account of the Berezin expectation is taken from [4]; its conventions lead to the involvement of antinormal ordering. It is possible to alter conventions so as to involve normal ordering instead: in fact, reversal of the two-form $\gamma \in \bigwedge^{2 M} (V_{\CC})$ leads to a new Berezin expectation that satisfies 
$$\mathbb{E} (\gamma^+ (\eta) \wedge \gamma^- (\xi)) = \langle \xi | \eta \rangle_J = \tau (\gamma^+ (\eta)  \gamma^- (\xi))$$
while the new version of normal ordering  
$$\bigwedge (V_{\CC}) \ra \bigwedge (V_J^+) \otimes \bigwedge (V_J^-) = C (V_J^+) \otimes C (V_J^-) \ra C(V_{\CC})$$ 
replaces the formula in Theorem \ref{nor3} by 
$$\nu (x \wedge y \wedge z) = x y z - \langle y | z \rangle x + \langle x | z \rangle y - \langle x | y \rangle z$$ 
and the formula in Theorem \ref{nor2} by 
$$\nu (x \wedge y) = x y - \langle x | y \rangle {\bf 1}.$$

\medbreak 

Throughout the present paper, we have considered exclusively the case in which the vector space $V$ is finite-dimensional; the case in which $V$ is infinite-dimensional calls for more careful handling and will be considered in a subsequent paper. 

\medbreak 

\bigbreak

\begin{center} 
{\small R}{\footnotesize EFERENCES}
\end{center} 
\medbreak 

[1] F.A. Berezin, {\it The Method of Second Quantization}, Academic Press (1966).  

\medbreak 

[2] W.H. Greub, {\it Multilinear Algebra}, Springer-Verlag (1967). 

\medbreak 

[3] R.J. Plymen and P.L. Robinson, {\it Spinors in Hilbert Space}, Cambridge Tracts in Mathematics {\bf 114}, Cambridge University Press (1994). 

\medbreak 

[4] P.L. Robinson, {\it The Berezin Calculus}, Publ. RIMS, Kyoto University {\bf 35} (1999) 123-194. 

\medbreak

[5] V.S. Varadarajan, {\it Supersymmetry for Mathematicians: An Introduction}, Courant Lecture Notes in Mathematics {\bf 11}, American Mathematical Society (2004).

\medbreak 

\medbreak

\end{document}